\documentclass[12pt, oneside]{amsart}

\usepackage{amsmath,ifthen, amsfonts, amssymb, 
srcltx, 
   amsopn,tikz, tkz-euclide, url}
\usetikzlibrary{decorations.markings,arrows, cd}
\usepackage{tikz-cd}
 \usepackage{dutchcal}
\usepackage{xfrac} 
\usepackage{graphicx, ifsym}
\usepackage{mathtools}
\usepackage{fullpage}
\usepackage{bm,amsmath,amssymb,color,subfigure}
\usepackage{fancyhdr, ifpdf}

\newcommand{\showcomments}{yes}
\renewcommand{\showcomments}{no}

\newsavebox{\commentbox}
\newenvironment{com}%
{\ifthenelse{\equal{\showcomments}{yes}}%
{\footnotemark
        \begin{lrbox}{\commentbox}
        \begin{minipage}[t]{1.25in}\raggedright\sffamily\tiny
        \footnotemark[\arabic{footnote}]}
{\begin{lrbox}{\commentbox}}}%
{\ifthenelse{\equal{\showcomments}{yes}}%
{\end{minipage}\end{lrbox}\marginpar{\usebox{\commentbox}}}
{\end{lrbox}}}

\usepackage{amsmath,amsfonts, amssymb, graphicx, tikz, bbding,ifthen,srcltx,amsopn, url}
\usepackage{bm}
\usetikzlibrary{decorations.markings,arrows, intersections}
\usetikzlibrary{decorations.pathreplacing}

\DeclareMathOperator{\Aut}{Aut}

\DeclareMathOperator{\cat}{CAT}
\DeclareMathOperator{\bdry}{\partial_{\infty}}

\theoremstyle{definition}
\newtheorem{thm}{Theorem}[section]
\newtheorem{lem}[thm]{Lemma}
\newtheorem{prop}[thm]{Proposition}

\newtheorem{exa}[thm]{Example}

\newtheorem{cor}[thm]{Corollary}

\newtheorem*{assumption}{Standing Assumption}

\author{Kasia Jankiewicz}
\author{Annette Karrer}
\author{Kim Ruane}
\author{Bakul Sathaye}
\address{Department of Mathematics, University of California, Santa Cruz, CA 95064, USA}
\email{kasia@ucsc.edu}
\address{Department of Mathematics and Statistics - McGill University, 805 Sherbrooke Street West
Montreal, Canada}
\email{annette.karrer@mail.mcgill.ca}
\address{Department of Mathematics, Tufts University, Medford, MA  02155, USA}
\email{kim.ruane@tufts.edu}
\address{Mathematisches Institut, WWU M\"unster, 48149 M\"unster, Germany}
\email{bsathaye@uni-muenster.de}

\title{The boundary rigidity of lattices in products of trees}
\begin{document}
\begin{com}
{\bf \normalsize COMMENTS\\}
ARE\\
SHOWING!\\
\end{com}

\begin{abstract} We show that every group acting properly and cocompactly by isometries on a product of $n$ bounded valance, bushy 
 trees is boundary rigid. That means that every CAT(0) space that admits a geometric action of any such group has the visual boundary homeomorphic to a join of $n$ copies of the Cantor set.
\end{abstract}
\maketitle
\section{Introduction}
A visual boundary is a particular type of compactification of a proper CAT(0) metric space. The boundary is defined as a set of equivalence classes of asymptotic rays endowed with an appropriate topology. For hyperbolic spaces $X$ and $Y$, any quasi-isometry $X\to Y$ between them extends to a homeomorphism of their visual boundaries. Consequently, the homeomorphism type of the boundary of a hyperbolic group  is a well-defined group invariant. This is not true for $\cat(0)$ groups, i.e.\ groups that act geometrically on $\cat(0)$ spaces. 

Bowers-Ruane give an example of a group $G$ acting geometrically on $\cat(0)$ spaces $X$ and $Y$, such that the associated $G$-equivariant quasi-isometry between the spaces does not extend to a homoemorphism between their visual boundaries \cite{BowersRuane96}. Croke-Kleiner provided an example of a $\cat(0)$ group $G$ and two $\cat(0)$ spaces $X,Y$, both admitting geometric actions by $G$ such that $\bdry X$ and $\bdry Y$ are non-homeomorphic \cite{CrokeKleiner2000}. Wilson further showed that in fact this same $G$ acts geometrically on uncountably many spaces with boundaries of distinct topological type \cite{Wilson2005}. The group $G$ in the Croke-Kleiner construction is the right-angled Artin group with the defining graph a path on four vertices. 

A $\cat(0)$ group $G$ is called \emph{boundary rigid} if the visual boundaries of all $\cat(0)$ spaces admitting a geometric action by $G$ are homeomorphic. As noted above, hyperbolic $\cat(0)$ groups are boundary rigid while not all $\cat(0)$ groups are boundary rigid.  Ruane proved that the direct product of hyperbolic groups is boundary rigid \cite{Ruane99}. Hosaka extended that to show that any direct product of boundary rigid groups is boundary rigid \cite{Hosaka2003}. Hruska-Kleiner proved that groups acting geometrically on $\cat(0)$ spaces with the isolated flats property are boundary rigid \cite{HruskaKleiner05}.

A \emph{bushy tree} is a simplicial tree that is not quasi-isometric to a point or a line. 
In this note, we study a family of $\cat(0)$ groups acting geometrically on the product of $n$ bushy trees. 
We will assume the groups preserve the factors, which is always the case after passing to a finite-index subgroup.
We refer to such groups as \emph{lattices in a product of trees}. We skip the word \emph{cocompact} even though we are assuming this property throughout the paper.

The simplest example of a lattice in a product of trees is a direct product $F_n\times F_m$ of two finite rank free groups. These are boundary rigid by \cite{Ruane99}.  However, there exist lattices in product of trees that are \emph{irreducible}, i.e.\ they do not split as direct products, even after passing to a finite index subgroups.   Irreducible lattices in products of trees were first studies by Mozes \cite{Mozes92}, Burger-Mozes \cite{BurgerMozes97}, \cite{BurgerMozes2000} and by \cite{Wise96Thesis}. Burger-Mozes constructed examples of simple lattices in a product of trees, providing the first examples of simple $\cat(0)$ groups, as well as the first examples of simple amalgamated products of free groups.

In this paper, we prove the following. 

\begin{thm}\label{thm:main}
Let $G$ be a lattice in a product of $n$ trees. Suppose $G$ acts geometrically on a $\cat(0)$ space $X$. Then $\bdry X$ is the join of $n$ Cantor sets. 
\end{thm}

\begin{cor}
Every lattice in a product of $n$ trees  is boundary rigid.
\end{cor}

Upon the revision of the paper, we learned that Margolis proved a much stronger result than Theorem~\ref{thm:main} \cite[Thm K]{Margolis22}. He proves boundary rigidity for finitely generated groups quasi-isometric to a product of finitely many proper cocompact non-elementary hyperbolic spaces.

There are two steps in our proof of Theorem~\ref{thm:main}.  We first deduce from the rank-rigidity results of Monod~\cite{Monod06} and Caprace--Monod~\cite{CapraceMonod09} that we may assume without loss of generality that $X$ splits as a product of $n$ convex subspaces. 	
	This will imply that  $\bdry X$ splits as a join of $n$ $0$-dimensional topological spaces. In the special case when $X$ is a CAT(0) cube complex and $X$ is either geodesically complete or the action of $G$ on $X$ is essential, we can use rank rigidity for CAT(0) cube complexes, due to Caprace-Sageev \cite{CapraceSageev2011}, to obtain a necessary splitting. We mention this case separately to illustrate the usefulness of the Caprace-Sageev Rank Rigidity Theorem in the cube complex setting.

 The second step is to prove that the $n$ $0$-dimensional subspaces obtained after the first step are indeed homeomorphic to Cantor spaces.

The paper is organized as follows.  In section ~\ref{sec:background} we give the background on lattices in products of trees as well as ends and boundaries of $\cat(0)$ spaces.  We prove the main theorem in section~\ref{sec:meinthm}.

\subsection*{Acknowledgements}
We would like to thank Nicolas Monod for his helpful comments. 
This collaboration was facilitated by WIGGD 2020, which was sponsored by the NSF under grants DMS–1552234, DMS–1651963, and DMS–1848346.
The first author was also supported by the NSF grant DMS-2105548/2203307, the second author by the Israel Science Foundation grant no.1562/19, and the last author by the Deutsche Forschungsgemeinschaft (DFG, German Research Foundation) under Germany's Excellence Strategy EXC 2044 –390685587, Mathematics M\"unster: Dynamics–Geometry–Structure.

\section{CAT(0) spaces at infinity}\label{sec:background}
\subsection{Ends of a space}
We recall the definitions and relevant facts about the space of ends of a topological space. For more details, see \cite{BridsonHaefliger}.

Let $X$ be a topological space. A \emph{ray} in $X$ is a proper map $r:[0,\infty)\to X$. A \emph{ray at $x_0$} where $x_0$ is a point of $X$, is a ray with $r(0) = x_0$.
An \emph{end} $e$ of $X$ is an equivalence class of rays in $X$ where $r_1\simeq r_2$ if and only if for every compact set $K\subseteq X$ there exists $N\geq 0$ such that $r_1([N,\infty])$ and $r_2([N, \infty])$ are contained in the same connected component of $X-K$. We denote the equivalence class of the ray $r$ by $e(r)$. 
The set of all ends of $X$ is denoted by $Ends(X)$.

Let $U$ be an open set in $X$ and $e\in Ends(X)$. We use the notation $e< U$ to  mean that for any $r:[0,\infty)\to X$ with $e(r)=e$, there exists $N\ge 0$ such that $r([N,\infty))\subseteq  U$. 

The set $X\cup Ends(X)$, denoted by $\widehat X$, can be endowed with topology that is generated by the basis consisting of the following sets:
\begin{itemize}
    \item open sets in $X$,
    \item sets of the form $U\cup \{e\in Ends(X)\ | e < U\}$ where $U$ is a connected component of $X-K$ for some compact set $K\subseteq X$.
\end{itemize}

The space $\widehat X$ is compact and is called the \emph{end compactification} of $X$.

\subsection{Visual Boundary}
Assume that $X$ is a metric space with metric $d$. Two geodesic rays $r,r'$ are \emph{asymptotic}, if there exists a constant $K>0$ such that $d(r(t), r'(t))<K$ for all $t\in[0,\infty)$.

The \emph{boundary} of $X$, denoted $\partial X$, is the set of equivalence classes of geodesic rays, where two rays are equivalent if they are asymptotic. We denote the equivalence class of a ray $r$ by $r(\infty)$.

When $X$ is a complete CAT(0) space, we can put a topology on $\partial X$ as follows. First fix a basepoint $x_0\in X$.  
The \emph{cone topology} on $\partial X$ with respect to $x_0$ is given by the neighborhood basis $\{U(r, R, \epsilon): r(\infty)\in \partial X,\, R,\epsilon>0\}$ where 
$$U(r,R,\epsilon) = \{r'(\infty)\in \partial X : r'(0) = x_0, d(r(R), r'(R))<\epsilon \}.$$

This topology seems to depend on the choice of basepoint $x_0\in X$ but in fact it does not.  
There is a well-defined change of basepoint homeomorphism between the topologies determined by different basepoints. This follows from the fact:

\begin{prop}\label{prop: change of basepoint}\cite[Prop. II.8.2]{BridsonHaefliger} If  $r$ is a geodesic ray based at $x$ in a complete CAT(0) space $X$ and $x'$ is a point not on this ray, then there exists a unique geodesic ray $r'$ with $r'(0)=x'$ that is asymptotic to $r$.
\end{prop}

The boundary $\partial X$ endowed with the cone topology is called the \emph{visual boundary} of $X$ and we denote it by $\partial_{\infty} X$. 
See \cite[Chap II.8]{BridsonHaefliger} for more details and properties of the visual boundary. 

If $X$ is a proper CAT(0) space, then $\bdry X$ is  compact and  there is a natural well-defined map $\partial_{\infty}X\to Ends(Y)$ sending a ray $r$ to $e(r)$. 
This map does not depend on the choice of a ray in the equivalence class of asymptotic rays. The map is a continuous surjection.  \cite[Rem II.8.10]{BridsonHaefliger}.

One important theorem we will use about the visual boundary is the following theorem of Geoghegan-Ontaneda.

\begin{thm}[{\cite{GeogheganOntaneda2007}}]\label{dimension qi invariant}
The topological dimension of $\partial_{\infty} X$ is a quasi-isometry invariant.  
In particular, if a group $G$ acts geometrically on CAT(0) spaces $X$ and $X'$, then their visual boundaries have the same topological dimension.
\end{thm}

We will also use the following Lemma due to Papasoglu-Swenson.
\begin{lem}[{\cite[Lem 26]{PapasogluSwenson09}}]\label{lem:[papasogluSwenson}
Let $G$ be a group acting geometrically on a CAT(0) space $X$.
Then $G$ virtually stabilizes a finite subset $A$ of the visual boundary of $X$ if and only if $G$ virtually has $\mathbb Z$ as a direct factor. 
\end{lem}

\begin{proof} The contents of {\cite[Lem 26]{PapasogluSwenson09}}]\label{lem:[papasogluSwenson} proves that if $G$ virtually stabilizes a finite subset of $\bdry{X}$, then $G$ virtually has $\mathbb Z$ as a direct factor.  We prove the other direction here.  

Suppose $G$ virtually has $\mathbb Z$ as direct factor.  Thus there is a finite index subgroup of $G$ of the form $H\times\mathbb Z$.  Let $g$ be the generator of the virtual $\mathbb Z$ factor.  Notice that $C_G(g)$, the centralizer of $g$ in $G$, contains this subgroup and is therefore also finite index in $G$. Then for each $x\in G$, $x^n$ is in $C_G(g)$.  Every element in $C_G(g)$ fixes $\{g^{\pm\infty}\}$, the endpoints of $g$ in $\partial_{\infty} X$.  Thus every element of the finite index subgroup $H\times\mathbb Z$ fixes this finite set in the boundary as needed. 

\end{proof}

We do not want to assume $X$ is geodesically complete for our theorem so we use work of Caprace--Monod in \cite{CapraceMonod09}
to get around this issue.  We explain how to do this here. 

An action of a group $G$ on a CAT(0) space $X$ is \emph{minimal} if there does not exists a non-empty $G$-invariant closed convex subset of $X$. 
A CAT(0) space $X$ is \emph{minimal} if the action of its full isometry group is minimal.

A subset $X'\subseteq X$ is \emph{quasi-dense} in $X$ if there exists a $D>0$ such that each point of $X$ is within distance $D$ of $X'$.  In particular, $\partial X=\partial X'$ as sets.  
Thus, for purposes of studying the boundary, we can always pass to a quasi-dense subset without losing any information. The following theorem will be used to avoid the extra assumption of geodesic completeness.

\begin{thm}[{\cite{CapraceMonod09}, see also \cite[Ex II.4 and Prop III.10]{Caprace14Notes}}]\label{thm:CapraceMonod}
\label{thmCapraceMonod}
Let $G$ be a group acting geometrically on a CAT(0) space $X$. 
Then $G$ stabilizes a closed, convex, quasi-dense subspace $X'\subseteq X$ such that $G$ acts minimally on $X'$.
In particular, $\partial X =\partial X'$ as sets. 

Moreover, $\partial_T X$ splits as a join if and only if $X'$ splits as a product $X_1\times X_2$.
\end{thm}
The subspace $X'$ in the above theorem is not necessarily unique. 
In fact, every closed, convex, quasi-dense subspace of $X$ on which $G$ acts minimally admits a splitting as a product, provided that $\partial_T X$ splits as a join.

\section{Lattices in product of trees}\label{sec:lattices}
In this section, we summarize facts about lattices in products of trees used in the proof of Theorem~\ref{thm:main}. A simplicial tree $T$ is \emph{bushy}, if it is not quasi-isometric to a point or a line. For example, a regular tree of valence $\geq 3$ is bushy. 

\begin{assumption}
Throughout the paper we assume that all the trees we consider are bushy trees of bounded valence.
\end{assumption}
We view $T$ as a metric space, with the path metric, where each edge has length $1$. The automorphism group $\Aut(T)$ is a group of all isometries $T\to T$, i.e.\ permutations of the vertex sets that preserve the adjacency. The group $\Aut(T)$ endowed with a compact-open topology, is a locally compact group. 
We now consider $n$ trees $T_1, \dots, T_n$. The product $T_1\times\dots\times T_n$ has a natural structure of a cube complex. It is easy to verify that each vertex link  is a complete $n$-partite graph and, in particular, it is a flag simplicial complex. Thus $T_1\times\dots\times T_n$, with the path metric induced by the Euclidean metric on each cube, is a CAT(0) cube complex.

\begin{exa}
Let $G$ be a group on four generators $a,b,x,y$ with four relations:
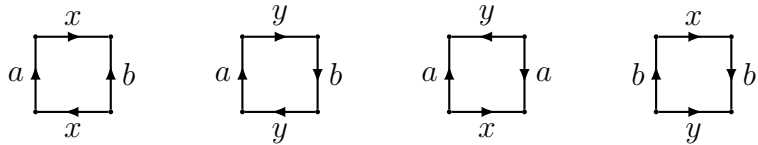
\begin{center}
 \begin{tikzpicture}
\tikzset{>=latex}
\begin{scope}[thick, decoration={
    markings,
    mark=at position 0.62 with {\arrow{latex}}}
    ] 
\node[circle, draw, fill, inner sep = 0pt,minimum width = 1pt] (a) at (0,0) {};
\node[circle, draw, fill, inner sep = 0pt, minimum width = 1pt] (b) at (0,1) {};
\node[circle, draw, fill, inner sep = 0pt, minimum width = 1pt] (c) at (1,1) {};
\node[circle, draw, fill, inner sep = 0pt, minimum width = 1pt] (d) at (1,0) {};
\draw[postaction={decorate}] (a) -- node[left] {$a$}  (b);
\draw[postaction={decorate}] (b) --node[above] {$x$} (c);
\draw[postaction={decorate}] (d) --node[right] {$b$}  (c);
\draw[postaction={decorate}] (d) --node[below] {$x$} (a);
\end{scope}
\begin{scope}[shift = {(2.75,0)}, thick, decoration={
    markings,
    mark=at position 0.62 with {\arrow{>}}}
    ] 
\node[circle, draw, fill, inner sep = 0pt,minimum width = 1pt] (a) at (0,0) {};
\node[circle, draw, fill, inner sep = 0pt, minimum width = 1pt] (b) at (0,1) {};
\node[circle, draw, fill, inner sep = 0pt, minimum width = 1pt] (c) at (1,1) {};
\node[circle, draw, fill, inner sep = 0pt, minimum width = 1pt] (d) at (1,0) {};
\draw[postaction={decorate}] (a) -- node[left] {$a$}  (b);
\draw[postaction={decorate}] (b) --node[above] {$y$} (c);
\draw[postaction={decorate}] (c) --node[right] {$b$}  (d);
\draw[postaction={decorate}] (d) --node[below] {$y$} (a);
\end{scope}
\begin{scope}[shift = {(5.5,0)},thick, decoration={
    markings,
    mark=at position 0.62 with {\arrow{>}}}
    ] 
\node[circle, draw, fill, inner sep = 0pt,minimum width = 1pt] (a) at (0,0) {};
\node[circle, draw, fill, inner sep = 0pt, minimum width = 1pt] (b) at (0,1) {};
\node[circle, draw, fill, inner sep = 0pt, minimum width = 1pt] (c) at (1,1) {};
\node[circle, draw, fill, inner sep = 0pt, minimum width = 1pt] (d) at (1,0) {};
\draw[postaction={decorate}] (a) -- node[left] {$a$}  (b);
\draw[postaction={decorate}] (c) --node[above] {$y$} (b);
\draw[postaction={decorate}] (c) --node[right] {$a$}  (d);
\draw[postaction={decorate}] (a) --node[below] {$x$} (d);
\end{scope}
\begin{scope}[shift = {(8.25,0)},thick, decoration={
    markings,
    mark=at position 0.62 with {\arrow{>}}}
    ] 
\node[circle, draw, fill, inner sep = 0pt,minimum width = 1pt] (a) at (0,0) {};
\node[circle, draw, fill, inner sep = 0pt, minimum width = 1pt] (b) at (0,1) {};
\node[circle, draw, fill, inner sep = 0pt, minimum width = 1pt] (c) at (1,1) {};
\node[circle, draw, fill, inner sep = 0pt, minimum width = 1pt] (d) at (1,0) {};
\draw[postaction={decorate}] (a) -- node[left] {$b$}  (b);
\draw[postaction={decorate}] (b) --node[above] {$x$} (c);
\draw[postaction={decorate}] (c) --node[right] {$b$}  (d);
\draw[postaction={decorate}] (a) --node[below] {$y$} (d);
\end{scope}
\end{tikzpicture}
\end{center}
The group $G$ is an irreducible lattice in the product of two copies of a $4$-valent tree \cite{JanzenWise09}.
\end{exa}

The presentation complex of $G$ in the example above, and for all torsion-free lattices in product of two trees more generally, is a non-positively curved square complex where each vertex link is a complete bipartite graph. 
Such square complexes are referred to as \emph{complete square complexes}. 
We refer the reader to \cite{WiseCSC} for the theory of complete square complexes.
Our approach in the $n$-factor case is inspired by Wise's work in the two factor case. 

Let $G$ be a lattice in  $\Aut(T_{1})\times\dots \times \Aut(T_n)$, and let $Z$ be the orbispace obtained as the quotient of $T_{1}\times\dots \times T_n$ by the action of $G$.
Let $q:T_{1}\times\dots \times T_n \to Z$ denotes the quotient map, which is an orbispace covering map. 
Let $v\in T_1\times\dots\times T_n$ be a basepoint. We identify $G$ with the orbispace fundamental group $\pi_1(Z,\overline v)$ where $\overline v = q(v)$.

All the edges of the product $T_{1}\times\dots \times T_n$ can be labelled by $1,\dots, n$ indicating in which factor the edge lies. 
Since the action of $G$ on $T_{1}\times\dots \times T_n$ preserves the factors, the underlying space of $Z$ is a $n$-dimensional cube complex whose edges can also be labelled by $1,\dots, n$ according to the label on the lift of the edge. 
Consider the subspace of $Z$ consisting of all the edges with label $i$, and denote its connected component containing $\overline v$ by $Z_i$. Note that $Z_i$ is a $1$-dimensional cube complex, i.e.\ $Z_i$ is a graph.

For each $i=1$, let $pr_i: T_1\times\dots\times T_n
\to T_1\times\dots\times \widehat{T_i}\times\dots\times T_n$ be the projection onto $n-1$ factors where $\hat{\cdot}$ denotes an omission of a factor.
Similarly, let $\rho_i:G\to \Aut(T_1)\times\dots\times \widehat\Aut({T_i})\times\dots\times \Aut(T_n)$ denote the projected action of $G$ on $T_1\times\dots\times \widehat{T_i}\times\dots\times T_n$.
For each $i= 1,\dots, n$ we define a subgroup $G_i\subseteq G$ as the stabilizer of $pr_i(v)$ with respect to the action $\rho_i$, i.e.
\begin{equation}
G_i = Stab_{T_1\times\dots\times \widehat{T_i}\times\dots\times T_n}(pr_i(v)).
\tag{$*$}\label{eq:G_i}\end{equation}

In particular, $G_i$ fixes setwise a copy of $T_i$ in the original action of $G$ on $T_1\times \dots \times T_n$. The coordinates other than the $i$-th one of this copy of $T_i$ are the coordinates of the basepoint $v$ of $T_1\times \dots \times T_n$. We abuse notation to denote that copy of $T_i$ by $T_i\times \{pr_i(v)\}$.

We will prove that the action of $G_i$ on $T_i\times \{pr_i(v)\}$ is geometric. We first need the following lemma.

\begin{lem}\label{lem:torsion groups acting on trees}
Let $T$ be an unbounded simplicial tree and let $H$ be a group acting geometrically on $T$. Then $H$ contains an infinite order element.
\end{lem}
\begin{proof}
Suppose that $H$ contains no infinite order  elements. Every finite order isometry of a tree has a fixed point. There cannot exist two elements of $H$ with disjoint sets of fixed points, as their product would have infinite order. Thus the group $H$ has a global fixed point. By the properness of the action, $H$ must be finite, but this is impossible since $T$ is unbounded and the quotient of $T$ by the action of $H$ is finite.
\end{proof}

\begin{lem}\label{lem:action of Gi}
For each $i=1, \dots, n$, the group $G_i$ acts geometrically on $T_i\times \{pr_i(v)\}$, and contains an infinite order element.
\end{lem}

\begin{proof}
We will prove the statement for $i=1$.
The proof is analgous for every $i=1,\dots, n$.

The properness of the action of $G_i$ on $T_i\times \{pr_i(v)\}$ follows from \cite[Ch. II, Thm 6.10]{BridsonHaefliger}. 
In order to show that the action is cocompact, we first consider the restriction $T_1\times \{pr_1(v)\} \to Z_1$ of the quotient map $q$, which is also an orbispace covering map. Note that $Z_1$ is a finite graph, as a subspace of a finite cube complex, and $T_1\times \{pr_1(v)\}$ is the universal cover of $Z_1$. In particular, the action of  $\pi_1(Z_1, \overline v)$ on $T_1\times \{pr_1(v)\}$ is cocompact.
Moreover, by Lemma~\ref{lem:torsion groups acting on trees}, the orbispace fundamental group $\pi_1(Z_1, \overline v)$ has an infinite order element.

Let us now show that $\pi_1(Z_1, \overline v)$ embeds as a subgroup of $G_1$. First we check that the map between the orbispace fundamental groups $\pi_1(Z_1, \overline v)\to \pi_1(X, \overline v)$ induced by the inclusion $Z_1\subseteq X$ is injective. Take $g\in \pi_1(Z_1, \overline v)$, and denote its image in $\pi_1(X, \overline v)$ by $\hat g$. We can view both $\hat g$ and $g$ as Deck transformation of $T_1\times \dots \times T_n$ and $T_1\times \{pr_1(v)\}$ respectively. We have $g = p \circ\hat g \circ\iota$, where $\iota$ is the inclusion $\iota:T_1\times \{pr_1(v)\} \hookrightarrow T_1\times \dots \times T_n$, and $p: T_1\times \dots \times T_n\to T_1\times \{pr_1(v)\}$ is the projection onto the basepoint in the last $n-1$ trees. In particular, since $g$ is nontrivial, we deduce that $\hat g$ is nontrivial. 
This also shows that $\hat g$ maps $v$ to $g(v)$ whose last $n-1$ coordinates are $pr_1(v)$. Thus $\rho_1(\hat g)$ belongs to the stabilizer of $pr_1(v)$, and so $\hat g\in G_1$. Since the action of $\pi_1(Z_1, \overline v)$ on $T_1\times\{pr_1(v)\}$ is cocompact, and $\pi_1(Z_1, \overline v)\subseteq G_1$, we conclude that the action of $G_1$ on $T_1\times\{pr_1(v)\}$ is cocompact.
\end{proof}

\begin{lem}\label{lem: no Z factor}
A lattice $G$ in $\Aut(T_1)\times \dots \times \Aut(T_n)$ has no infinite order central elements. In particular, $G$ does not have a direct $\mathbb Z$ factor. 
\end{lem}
\begin{proof}
The second part follows directly from the first part, as if $G$ has a direct $\mathbb Z$ factors, then it contains an infinite order central element. 

Let $g$ be an infinite order central element. Then the points at infinity $g^{\pm\infty}$ are fixed by $G$, as $G$ centralizes $\langle g\rangle$. 
Since $G$ preserves the factors of $T_1\times\dots\times T_n$, the existence of a global fixed point in $\bdry T_1 *\dots * \bdry T_n$ of the action of $G$ implies that $G$ fixes a point in at least one of $\bdry T_i$. 
By Lemma~\ref{lem:action of Gi}, the group $G_i$ has no fixed point in $\bdry T_i$ since the action of $G_i$ on $T_i$ is geometric. In particular, $G$ has no fixed point in $\bdry T_i$. 
This completes the proof. 
\end{proof}

\section{Proof of Theorem~\ref{thm:main}}\label{sec:meinthm}

In order to prove the main theorem, we first conclude from Monod~\cite{Monod06} and Caprace--Monod~\cite{CapraceMonod09} that $X$ splits as a metric product. Afterwards, we show that the boundary of each factor of the metric product is a Cantor set.

\subsection{Splitting}\label{sec:splittingNEW}
	Let $G$ be a lattice in a product of $n$ trees acting geometrically on a CAT(0) space $X$. Recall that we may assume that $G$ acts minimally on $X$ by Theorem~\ref{thmCapraceMonod}. We conclude the following theorem from the rank-rigidity results of Monod and Caprace--Monod:
	
	\begin{thm}[\cite{Monod06,CapraceMonod09}]\label{thm:join}
	 	Let $G<\Aut(T_1) \times \dots \times\Aut(T_n)$ be an irreducible lattice in a product   of $n$  bushy trees $T_1,\dots,T_n$. 
	 	Suppose $G$ acts geometrically and minimally by on a CAT(0) space $X$. 
	 Then $X$ splits as a metric product of $n$ unbounded factors $X_1
  \times \dots \times X_n$, and the action of $G$ preserves the decomposition of $X$.
\end{thm}

\begin{proof}
By Theorem~\ref{dimension qi invariant}, $X$ has finite-dimensional boundary. 
By Lemma~\ref{lem: no Z factor} $G$ does not virtually have $\mathbb Z$ as a direct factor. In particular, $X$ does not have a Euclidean factor by \cite[Thm 2(v)]{CapraceMonod19}.
By \cite[Lem 26]{PapasogluSwenson09} $G$ acts without a fixed point at infinity.
Thus, the rank-rigidity theorem \cite[Thm 8.4]{CapraceMonod09} implies that the $G$-action of $X$ extends to a continuous $\Aut(T_1) \times \dots \times\Aut(T_n)$-action by isometries.
By \cite[Thm 9]{Monod06} $X$ splits $\Aut(T_1) \times \dots \times\Aut(T_n)$-equivalently isometrically as a direct product $X_1\times\dots \times X_n$ of $\Aut(T_i)$-spaces $X_i$.

It remains to show that each factor $X_i$ is unbounded.
Let $g_i$ be an infinite order element of $G_i$, provided by Lemma~\ref{lem:action of Gi}. Since the action of $G$ on $X$ is geometric, each element of $G$ acts on $X$ as a semi-simple isometry, and in particular, $g_i$ acts as a hyperbolic isometry. By \cite[Prop 6.9]{BridsonHaefliger} $g_i$ acts as a semi-simple isometry of each factor $X_1, \dots, X_n$. As an element of $G_i$, $g_i$ acts elliptically on each of the factors $X_1, \dots, X_{i-1}, X_{i+1}, \dots, X_n$, so $g_i$ must act as a hyperbolic isometry on $X_i$. Thus $X_i$ contains an axis of $g_i$, and in particular, is unbounded.
\end{proof}

 \subsection{Analyzing the join factors}\label{sec:from splitting to splitting of cantor sets}

Let $G$ be a lattice in a product of $n$ bushy trees.
In this subsection, we complete the proof of the main theorem by showing the following proposition.

\begin{prop}\label{prop:factors are cantor sets}
	Let $G$ be a lattice in a product of trees .
	Suppose $G$ acts geometrically on a CAT(0) space $X$ where $\bdry X = \bdry X_1*\dots* \bdry X_n$ for unbounded, closed, convex subspaces $X_1$,$\dots$, $X_n$ in $X$.
	Then each $\bdry X_1$, $\dots$,$\bdry X_n$ is homeomorphic to the Cantor set and each subspace $X_i$ is quasi-isometric to a finite valence bushy tree.
\end{prop}

The subtlety involved here is that when the group $G$ is \emph{not} (virtually) a product of free groups, we do not have geometric group actions on $X_i$.  
Indeed, if $G$ is an irreducible lattice, then the projected action of $G$ on $X_i$ is cocompact but not proper.

\begin{lem}\label{lem:boundary=space of ends} 
Suppose $G$ acts geometrically on a $\cat(0)$ space $X$ so that $\bdry X=\bdry X_1*\dots*\bdry X_n$ for closed convex subspaces $X_1,\dots, X_n$ in $X$. If the topological dimension of $\partial X_1,\dots,\partial X_n$ are zero, then the boundary $\partial_{\infty} X_i$ is homeomorphic to $\textrm{Ends}(X_i)$, the ends space of $X_i$.
\end{lem}
\begin{proof} Let us first show that each $X_i$ is hyperbolic.  Note that since $X$ is proper, so are $X_i, \dots, X_n$.

First, we show that the subspace $X_i$ for $i=\{1,\dots,n\}$ is a \emph{visibility space}, i.e. given any two distinct points $\xi, \eta$ in $\bdry X_i$, there is a geodesic line in $X_i$ between them.
Suppose there is no geodesic line in $X_i$ between $\xi$ and $\eta$. Then by \cite[Prop. II.9.21(2)]{BridsonHaefliger}, there is a geodesic segment in $\partial_T X_i$ joining them. This segment is an arc in $\partial_T X_i$ which would map to an arc in $\bdry X_i$ via the identity map on $\partial X$.  This contradicts the fact that $\bdry X_i$ is 0-dimensional. 

Moreover, the action of $G$ on $X_i$ is cocompact, since so is the action of $G$ on $X$. Hence, $X_i$ 
is uniformly visible by \cite[Prop II.9.32]{BridsonHaefliger}. Finally, $X_i$ is hyperbolic by \cite[Prop III.1.4]{BridsonHaefliger}.

Since $X_i$ is a proper hyperbolic space, the natural map $\bdry X_i\to \textrm{Ends}(X_i)$ is continuous and the fibers of that map are the connected components of $\bdry X_i$ \cite[Exer III.H.3.9]{BridsonHaefliger}. 
Since $\bdry X_i$ is $0$-dimensional, the connected components are single points. 
Thus the map $\bdry X_i\to \textrm{Ends}(X_i)$ is a continuous bijection. 
Every continuous bijection from a compact space to a Hausdorff space is a homeomorphism.
\end{proof}

In \cite[Def 1.1]{Bestvina96}, Bestvina outlined a set of axioms that a group boundary should have in order to be useful for relating homological invariants of the boundary to cohomological invariants of the group. All of the axioms hold true for a hyperbolic group $G$ acting on $G\cup\bdry G$ and for a $\cat(0)$ group $G$ acting on $X\cup\bdry X$ where $G$ admits a geometric action on the $\cat(0)$ space $X$. One of the axioms requires the collection of translates of any compact set to form a null set in $X\cup\bdry X$, i.e. for any open cover $\mathcal U$ of $X\cup\bdry X$ and any compact set $K$ in $X$, all but finitely many $G$-translates of $K$ are contained in an element of $\mathcal U$. 

The next lemma shows that each $X_i\cup\bdry X_i$ inherits this nullity condition on compact sets from $X\cup\bdry X$ even though we have no geometric group action on $X_i$.  

\begin{lem}\label{lem:Zset property}
Suppose $G$ acts geometrically on a $\cat(0)$ space $X$ with $\bdry X=\bdry X_1*\dots *\bdry X_n$ for closed convex subspaces $X_1,\dots, X_n$ in $X$.
 For every compact set $K\subseteq X_i$ and every open neighborhood $U$ of an end of $X_i$, there exists $g\in G$ such that $gK\subseteq U$.
\end{lem}

\begin{proof} 
We show that the lemma holds for $X_1$. The argument for $X_i$, $i \neq 1$ is identical.
Let $K\subseteq X_1$ and $K'\subseteq X_2\times\dots\times X_n$ be compact sets. Then $K\times K'$ is a compact set in $X$. By Lemma~\ref{lem:boundary=space of ends}, $\bdry X_1$ is homeomorphic to $\textrm{Ends}(X_1)$. Let $\xi\in\bdry X_1$, and let $U$ be an open neighborhood of $\xi$. Then $U\times X_2\times\dots\times X_n$ is an open neighborhood of $\xi$, viewed as a point in $\bdry X= \bdry X_1* \bdry X_2*\dots*\bdry X_n$. Since the action of $G$ on $X$ is geometric, there exists $g\in G$ such that $g(K\times K')\subseteq U\times X_2\times \dots, \times X_n$ (see e.g.\ \cite[p.124]{Bestvina96}). It follows that in the action of $G$ on $X_1$, we have $gK\subseteq U$. \end{proof}

The following is not stated explicitly, but it is proved in \cite{Hopf44}. It is also proved in a similar form in \cite{KronNotes}. We include the proof for completeness. We restrict our attention to the case of geodesic metric spaces, but the proposition holds in more general setting.

\begin{prop}[\cite{Hopf44}]\label{prop:Hopf}
Let $Y$ be a geodesic metric space with at least three ends.
Let $G$ be a group acting on $Y$ cocompactly so that the following holds: 
for every compact set $K\subseteq Y$, and every open neighborhood $U$ of an end of $Y$, 
there exists $g\in G$ such that $gK\subseteq U$.
Then the space of ends $\textrm{Ends}(Y)$ is perfect.
\end{prop}
\begin{proof}
Suppose that there exists an end $e\in \textrm{Ends}(Y)$ that is isolated, i.e.\ there exists a neighborhood $U$ of $e$ that does not contain any other ends. 

First we show that without loss of generality, we can assume that $Y-U$ is connected. 
Indeed, if $Y-U$ is not connected, we construct a neighborhood $U'\subseteq U$ of $e$ such that $Y-U'$ is connected. 
Since $\widehat Y$ is compact, so is $\widehat Y-U$. 
Let $V_1, \dots, V_n$ be a finite collection of open sets covering $\widehat Y-U$ such that each $V_i$ is connected and does not contain $e$, and its closure $\widehat V_i$ in $\widehat Y$ is compact. 
The union $\bigcup_{i=1^n}V_i$ has finitely many components, and since $\textrm{Ends}(Y)$ is nowhere dense in $\widehat Y$, each $V_i$ contains points of $Y$. Each two points in $Y$ can be joined by a path in $Y$.
In particular, there exists a closed connected set $Q$ which is the union of $\bigcup_{i=1^n}V_i$ and a finite number of paths in $Y$. Note that $Q$ does not contain $e$. The set $U' = \widehat Y-Q$ is an open neighborhood of $e$ and since $Y-U\subseteq \bigcup V_i\subseteq Q$, we have $U'\subseteq U$. 
Thus, $U'$ is the neighborhood we were looking for.

By assumption, there are at least three distinct ends $e_1, e_2, e_3$ in $Y$. Let $K$ be a compact set in $Y$ such that each of the ends $e_1, e_2, e_3$ lies in a different connected component $Y_1, Y_2, Y_3$ of $Y-K$. By assumption, there exists an element $g\in G$ such that $gK\subseteq U$. We claim that for each $i=1,2,3$, either $gY_i\subseteq U$, or $g(Y-Y_i)\subseteq U$. Indeed, otherwise there exists point $p\in gY_i-U$ and $q\in g(Y-Y_i)-U$. Since $Y-U$ is connected, there exists a path $\gamma$ in $Y-U$ joining $p$ and $q$.
 Note that $p,q$ lie in distinct connected component of $Y-gK$, so $\gamma$ has to pass through $gK$. This is a contradiction, since $gK\subseteq U$.

Since $U$ contains only one end, we have $g(Y-Y_i)\subseteq U$ for at least two $i$'s among $1,2,3$, say $1$ and $2$.
It follows that $Y-U\subseteq gY_i$ for $i=1,2$. The subsets $Y_1$ and $Y_2$ are disjoint, and so are $gY_1$ and $gY_2$. This is a contradiction and therefore the space of $\textrm{Ends}(Y)$ has no isolated points.
\end{proof}

\begin{lem}\label{lem:boundary Cantor set}
Suppose $G$ acts geometrically on a $\cat(0)$ space $X$ so that $\bdry X=\bdry X_1*\dots*\bdry X_n$ for closed convex subspaces $X_1,\dots, X_n$ in $X$. Suppose that the topological dimensions of $\partial X_1,\dots,\partial X_n$ are zero. 
If $|\partial X_i|\geq 3$, then $\partial X_i$ is homeomorphic to the Cantor set $\mathcal C$. 
\end{lem}
\begin{proof}
Since $X_i$ is a proper CAT(0) space, the boundary $\partial X_i$ is compact and metrizable \cite{BridsonHaefliger}. 
As $\dim \bdry X_i =0$, the boundary $\bdry X_i$ is totally disconnected.
The action of $G$ on $X_i$ is cocompact, since so is the action of $G$ on $X$. By Lemma~\ref{lem:Zset property}, for every compact set $K\subseteq X_i$ and every open neighborhood $U\subseteq X_i$ of an end of $X_i$, there exists $g\in G$ such that $gK\subseteq U$. By Proposition~\ref{prop:Hopf}, $\textrm{Ends}(X_i)$ is a perfect space. This implies that $\bdry X_i$ is a perfect space because $\textrm{Ends}(X_i) = \bdry X_i$ by Lemma~\ref{lem:boundary=space of ends}. 
By the characterization of the Cantor set, as a non-empty, perfect, totally disconnected, compact metrizable space, we conclude that $\partial X_i$ is the Cantor set. 
\end{proof}

\begin{proof}[Proof of Propostion~\ref{prop:factors are cantor sets}]
By Theorem~\ref{dimension qi invariant}, the topological dimension of $\bdry X$ coincides with the topological dimension of $\bdry T_1* \dots *\bdry T_n = n-1$, so each $\bdry X_i$ is $0$-dimensional.

By Theorem~\ref{thm:join} the action of $G$ on $X = X_1\times\dots\times X_n$ preserves the factors, and in particular, $G$ fixes each $\bdry X_i$ setwise. 
If for some $i=1,\dots, n$ the set $\bdry X_i$ was finite, then $G$ would have a $\mathbb Z$ factor, by Lemma~\ref{lem:[papasogluSwenson}. 
Since this is not the case for $G$ by Lemma~\ref{lem: no Z factor}, 
we deduce that $\bdry X_i$ is infinite for each $i=1,\dots, n$.
By Lemma~\ref{lem:boundary Cantor set} each $\bdry X_i$ is homeomorphic to a Cantor set.
Consequently, $\bdry X$ is homeomorphic to join of n copies of a Cantor set. 
We now show that $X_i$ is quasi-isometric to a tree, following \cite{Chao15}. By \cite[Lem 5.7]{Chao15}, $X_i$ has the bottleneck property, i.e. there is some $\Delta > 0$ so that for all $x, y \in X_i$ there is a
midpoint $m = m(x,y)$ with $d(x,m) = d(y,m) = \frac{1}{2} d(x,y) $ and the property that any path from $x$ to $y$ must pass within less than $\Delta$ of the point $m$. By \cite[Thm 4.6]{Mannning05}, this property is equivalent to $X_i$ being quasi-isometric to a simplicial tree $T'_i$. Since $X$ is a proper space, so is each $X_i$ and thus $T'_i$ is finite valence.  The boundary $\partial X_i$ is homeomorphic to a Cantor set and thus $T'_i$ is bushy.
\end{proof}

\bibliographystyle{alpha}
\bibliography{kasia}

\end{document}